\documentclass[12pt, reqno]{amsart}
\usepackage{amscd,amssymb,verbatim}
\usepackage{amssymb,amsmath,amsthm,epsfig}
\usepackage{enumerate}

\allowdisplaybreaks

\newcommand{\N}{\mathbb{N}}

\newcommand{\R}{\mathbb{R}}

\def\sideremark#1{\ifvmode\leavevmode\fi\vadjust{\vbox
to0pt{\vss \hbox to 0pt{\hskip\hsize\hskip1em
\vbox{\hsize2cm\tiny\raggedright\pretolerance10000
\noindent#1\hfill}\hss}\vbox to8pt{\vfil}\vss}}}

\newcommand{\vp}{\varepsilon}

\newtheorem{thm}{Theorem}[section]
\newtheorem{lem}[thm]{Lemma}

\newtheorem{prop}[thm]{Proposition}
\theoremstyle{definition}
\newtheorem{defin}[thm]{Definition}
\newtheorem{ex}[thm]{Example}

\newcommand{\coo}{c_{00}}

\newcommand{\spa}{{\rm span}}

\def\sfrac#1#2{\kern.1em\raise.5ex\hbox{$#1$}
        \kern-.1em/\kern-.05em\lower.25ex\hbox{$#2$}}

\usepackage{color}

\title{Weaving Schauder frames}

\begin{document}

\author{Peter G. Casazza}
\address{Department of Mathematics\\ University of Missouri\\ Columbia, MO 65211}
\email{casazzap@mail.missouri.edu}

\author{Daniel Freeman}
\address{Department of Mathematics and Computer Science\\ Saint Louis University\\ St Louis, MO 63103}
\email{dfreema7@slu.edu}

\author{Richard G. Lynch}
\address{Department of Mathematics\\ University of Missouri\\ Columbia, MO 65211}
\email{rglz82@mail.missouri.edu}

\thanks{The first and third authors were supported by
NSF DMS 1307685; NSF ATD 1321779.  The second author was supported
by grant 353293 from the Simon's foundation.}

\begin{abstract}
We extend the concept
of weaving Hilbert space frames to the Banach space setting. Similar
to frames in a Hilbert space, we show that for any two approximate
Schauder frames for a Banach space, every weaving is an approximate
Schauder frame if and only if there is a uniform constant $C\geq 1$
such that every weaving is a $C$-approximate Schauder frame. We also
study weaving Schauder bases, where it is necessary to introduce two
notions of weaving. On one hand, we can ask if two Schauder bases
are woven when considered as Schauder frames with their biorthogonal
functionals, and alternatively, we can ask if each weaving of two
Schauder bases remains a Schauder basis. We will prove that these
two notions coincide when all weavings are unconditional, but otherwise
they can be different. Lastly, we prove two perturbation
theorems for approximate Schauder frames.
\end{abstract}

\maketitle

\section{Introduction}

The concept of \emph{weaving Hilbert space frames} was first introduced in \cite{BCGLL}. Two frames $\{x_j^0\}_{j \in J}$ and $\{x_j^1\}_{j \in J}$ for a Hilbert space $H$ are called \emph{woven} if there exist constants $0 < A \leq B$ so that the \emph{weaving} $\{x_j^0\}_{j \in \sigma} \cup \{x_j^1\}_{j \in \sigma^c}$ is a frame with bounds $A,B$ for every choice of partition $\sigma \subset J$.

Weaving frames is motivated by distributed signal processing. We can think of each $j \in J$ as a sensor or node, and for each one we measure a signal with either $x_j^0$ or $x_j^1$. Can a signal be recovered robustly regardless of how the measurements are taken at each node?  Another way to think of this is if $(x_j^0)_{j\in J}$ and $(x^1_j)_{j\in J}$ are two sensor networks then we are interested in whenever parts of one network can be used to replace parts of the other.  In \cite{BCGLL}, \cite{CL}, and \cite{DV} the sensors are modeled by inner product with a vector in a frame or Riesz basis for a Hilbert space.  In this paper, we extend the concept of weaving frames to the Banach space setting. We will model sensors by evaluation by linear functionals associated with an approximate Schauder frame or Schauder basis.  We will be specifically considering weaving two bases or frames together, but we note that all of our results can be easily extended to weaving any finite number of bases or frames.


We begin by recalling the definition of a frame for a Hilbert space.

\begin{defin}
A set of vectors $(x_{j})_{j\in J}$ in a Hilbert space $H$ is called a {\em frame} for $H$ if there exists positive constants $A$ and $B$ (called the lower and upper {\em frame bounds} respectively) such that
\begin{equation}\label{e:frame}
A\|x\|^2\leq \sum_{j\in J} |\langle x, x_i\rangle|^2 \leq
B\|x\|^2\quad\textrm{for all }x\in H.
\end{equation}
\end{defin}

Given a set of vectors $(x_j)_{j\in J}$ in a Hilbert space $H$, we
define an operator $S:H\rightarrow H$ (called the {\em frame
operator}) by $S(x) :=\sum_{j\in J} \langle x, x_i\rangle x_i$, for all
$x\in H$, assuming that the series converges.  The frame operator
can be used to give a simple characterization of frames. Indeed,
$(x_j)_{j\in J}$ is a frame if and only if the frame operator is
well defined, bounded, and has bounded inverse.  Furthermore,
suppose that $(x_j)_{j\in J}$ is a frame with frame operator $S$,
 optimal lower frame bound $A$, and optimal upper frame bound $B$.  Then
$\|S\|=B$ and $\|S^{-1}\|=A^{-1}$.

 The characterization of a frame in terms of the frame operator is a useful perspective when considering how to
 generalize frame theory to Banach spaces.  Frames have been generalized to
 Banach spaces in multiple ways. One way is to generalize the frame inequality \eqref{e:frame} as in  \cite{G, FG}, but the following definition
 of a Schauder frame and approximate Schauder frame instead generalizes the
 definition of the frame operator.  Schauder frames were first used in
 \cite{CDOSZ} with the goal of creating a procedure to represent
vectors using quantized coefficients and are a generalization of framings of Banach spaces which were introduced in \cite{CHL}. Approximate Schauder frames
were defined in \cite{FOSZ} and were used in the construction of a Schauder frame for $L_p(\R^d)$ with $p>2$ formed by translations of a single
function.

\begin{defin} Let $X$ be a Banach space with dual space $X^*$.
A sequence $(x_i,f_i)_{i=1}^\infty$ in $X\times X^*$ is called an
{\em approximate Schauder frame} for $X$ if the operator
$S:X\rightarrow X$ (called the {\em frame operator}) defined by
$S(x):=\sum_{i=1}^\infty f_i(x)x_i$ is well defined and is bounded
with bounded inverse.  We say that $(x_i,f_i)_{i=1}^\infty$ is a
$C$-approximate Schauder frame if $\|S\|,\|S^{-1}\|\leq C$.  We say
that $(x_i,f_i)_{i=1}^\infty$ is a {\em Schauder frame} if the frame
operator equals the identity operator.
\end{defin}

 An approximate Schauder frame is called {\em unconditional} if the series
$\sum f_{\pi(i)}(x)x_{\pi(i)}$ converges to $S(x)$ for every
permutation $\pi$ of $\N$.  For $C_u\geq 1$, an unconditional
approximate Schauder frame is called \emph{$C_u$-unconditional} if for all $x\in
X$ and all choices of $\vp_i\in\{-1,1\}$, $\|\sum_{i\in\N} \vp_i
f_i(x)x_i\|\leq C_u\|S(x)\|$.  For $C_s\geq 1$, an unconditional
approximate Schauder frame is \emph{$C_s$-suppression unconditional} if for
all $x\in X$ and all $\Gamma\subseteq\N$, $\|\sum_{i\in\Gamma}
f_i(x)x_i\|\leq C_s\|S(x)\|$.   As is the case for unconditional
Schauder bases, we have that an approximate Schauder frame is
unconditional if and only if it suppression unconditional and
$C_s\leq C_u\leq 2C_s$.

Note that if $(x_i,f_i)_{i=1}^\infty$ is an approximate Schauder
frame for a Banach space $X$ with frame operator $S$ then $(x_i,
(S^{-1})^*f_i)_{i=1}^\infty$ is a Schauder frame for $X$.  This is a
generalization of the often used result that if $(x_i)_{i=1}^\infty$
is a frame for a Hilbert space $H$ with frame operator $S$ then
$(S^{-1/2} x_i)_{i=1}^\infty$ is a Parseval frame for $H$.

 Our goal is to extend the notion of weaving frames in
 a Hilbert space to weaving frames and bases for Banach spaces and to prove the analogous results in this setting.
  This is a different generalization of weaving frames than given in \cite{DV}, where they instead consider weaving together infinitely many Hilbert space
 frames.

Since ordering plays a crucial role, we index the weavings with elements from $\{0,1\}^\N$, the set of functions from $\N$ to $\{0,1\}$.  Note that $\{0,1\}^\N$ is a compact Hausdorff space in the product topology. We will repeatedly make use of this fact.

\begin{defin}\cite[Definition 3]{BCGLL}
Let $(x^0_j)_{j=1}^\infty$ and $(x^1_j)_{j=1}^\infty$ be two  frames for a
Hilbert space $H$.   We say that $(x^0_j)_{j=1}^\infty$ and
$(x^1_j)_{j=1}^\infty$ are {\em woven} if there are positive constants
$A$ and $B$ such that for every function
$\sigma\in\{0,1\}^\N$ we have that
$(x^{\sigma(j)}_j)_{j=1}^\infty$ is a frame for $H$ with lower frame
bound $A$ and upper frame bound $B$.
\end{defin}

  We give the following natural generalization of weaving frames  for
Hilbert spaces to weaving approximate Schauder frames for Banach
spaces.

\begin{defin}
Let $(x^0_j,f^0_j)_{j=1}^\infty$ and $(x^1_j,f^1_j)_{j=1}^\infty$ be
two approximate Schauder frames for a Banach space $X$.  A sequence $(x^{\sigma(j)}_j,f^{\sigma(j)}_j)_{j=1}^\infty$  with $\sigma\in\{0,1\}^\N$ is called a {\em weaving} of
$(x^0_j,f^0_j)_{j=1}^\infty$ and $(x^1_j,f^1_j)_{j=1}^\infty$.
Given $C \geq 1$,  we say that $(x^0_j,f^0_j)_{j=1}^\infty$ and
$(x^1_j,f^1_j)_{j=1}^\infty$ are {\em $C$-woven} if every weaving is a $C$-approximate Schauder frame.  We say that $(x^0_j,f^0_j)_{j=1}^\infty$ and
$(x^1_j,f^1_j)_{j=1}^\infty$ are {\em woven} if they are $C$-woven for some $C \geq 1$.
\end{defin}

Note that this is a true generalization of the definition of woven
frames for Hilbert spaces as two frames  $(x^0_j)_{j=1}^\infty$ and
$(x^1_j)_{j=1}^\infty$ of a Hilbert space $H$ are woven if and only if the approximate Schauder frames
$(x^0_j,x^{0}_j)_{j=1}^\infty$ and $(x^1_j,x^{1}_j)_{j=1}^\infty$ are
woven (where we identify a vector $y\in H$ with the linear functional $x\mapsto \langle x,y\rangle$).  

A \emph{Riesz basis} for a Hilbert space $H$ is any sequence of vectors that is image of some orthonormal basis under an invertible operator (this is equivalent to being a semi-normalized unconditional basis for a Hilbert space). In \cite{BCGLL}, it is proven that if $(x^0_j)_{j=1}^\infty$ and $(x^1_j)_{j=1}^\infty$ are two Riesz bases for a Hilbert space $H$ then $(x^{\sigma(j)}_j)_{j=1}^\infty$ is a Riesz basis for $H$ for all $\sigma\in\{0,1\}^\N$ if and only if $(x^{\sigma(j)}_j)_{j=1}^\infty$ is a frame for $H$ for all $\sigma\in\{0,1\}^\N$. Thus, there is no need to distinguish between two Riesz bases having all weavings be Riesz bases and two Riesz bases having all weavings be frames.   However, we will show that these two notions are not always the same for Schauder bases and thus we will require a second definition of weaving in this context.

A sequence $(x_j)_{j=1}^\infty$ in a Banach space $X$ is a {\em
Schauder basis} for $X$ if for every $x\in X$ there exists a unique
sequence of scalars $(a_j)_{j=1}^\infty$ such that
$x=\sum_{j=1}^\infty a_j x_j$, and a basis is called {\em
unconditional} if the series converges in every order for every
$x\in X$.  The {\em biorthogonal functionals} of
$(x_j)_{j=1}^\infty$ is the sequence $(x^*_j)_{j=1}^\infty$ in $X^*$
such that $x^*_j(x)=a_j$ for all $j\in \N$.   A sequence
$(x_j)_{j=1}^\infty$ in a Banach space $X$ is called a {\em basic
sequence} if it is a basis of its closed span $[x_j]_{j=1}^\infty$.
The {\em basis constant} of a basic sequence $(x_j)_{j=1}^\infty$ is
the least constant $C\geq 1$ such that $\|\sum_{i=1}^n a_i x_i\|\leq
C\|\sum_{i=1}^N a_i x_i \|$  for all choices of scalars
$(a_i)_{i=1}^N$ and  all $n\leq N$.

\begin{defin}
Let $(x^0_j)_{j=1}^\infty$ and $(x^1_j)_{j=1}^\infty$ be
two Schauder bases for a Banach space $X$.  A sequence $(x^{\sigma(j)}_j)_{j=1}^\infty$  with $\sigma\in\{0,1\}^\N$ is called a {\em weaving} of
$(x^0_j)_{j=1}^\infty$ and $(x^1_j)_{j=1}^\infty$.
We say that $(x^0_j)_{j=1}^\infty$ and
$(x^1_j)_{j=1}^\infty$ are {\em woven} if every weaving is a
Schauder basis for $X$.
\end{defin}

Let $(x^0_j)_{j=1}^\infty$ and $(x^1_j)_{j=1}^\infty$ be
two  Schauder bases for a Banach space $X$ with biorthogonal functionals  $(x^{0*}_j)_{j=1}^\infty$ and $(x^{1*}_j)_{j=1}^\infty$.  This implies that  $(x^0_j,x^{0*}_j)_{j=1}^\infty$ and $(x^1_j,x^{1*}_j)_{j=1}^\infty$
are both Schauder frames for $X$.  We can now consider whether or not the two bases  $(x^0_j)_{j=1}^\infty$ and $(x^1_j)_{j=1}^\infty$ are woven and whether or not the two Schauder frames
$(x^0_j,x^{0*}_j)_{j=1}^\infty$ and $(x^1_j,x^{1*}_j)_{j=1}^\infty$ are woven.  In Example \ref{E:conditional weaving}, we show that it is possible for two Schauder bases to be woven but for the corresponding Schauder frames to not be woven.  The key idea for why this is possible is that when we weave approximate Schauder frames we are weaving both two sequences of vectors and two sequences of functionals, but if we weave two Schauder bases and obtain a basis  the new corresponding biorthogonal functionals can be completely unrelated to the biorthogonal functionals of the two bases that we started with.  This gives us more flexibility when weaving Schauder bases, but it is actually a problem when we consider our motivation as determining when different sensor networks can be combined.  In that context, the given linear functionals are fundamental components of the sensors and thus we cannot change them based on what weaving we choose to use.  For this reason, it is important for applications for two Schauder bases to be woven both as bases and as approximate Schauder frames.  We prove that this is the case when all weavings are unconditional.  In particular,
 in Theorem \ref{T:unc} we prove a number of equivalent properties, including that if $(x^0_j)_{j=1}^\infty$ and $(x^1_j)_{j=1}^\infty$ are two unconditional Schauder bases then every weaving of $(x^0_j)_{j=1}^\infty$ and $(x^1_j)_{j=1}^\infty$ is an unconditional basis if and
only if every weaving of $(x^0_j,x^{0*}_j)_{j=1}^\infty$ and $(x^1_j,x^{1*}_j)_{j=1}^\infty$
 is an unconditional approximate Schauder frame.   In Theorem \ref{T:unc} we prove as well that  if $(x^0_j)_{j=1}^\infty$ and $(x^1_j)_{j=1}^\infty$ are two unconditional Schauder bases of a Banach space $X$ then every weaving of $(x^0_j)_{j=1}^\infty$ and $(x^1_j)_{j=1}^\infty$ is an unconditional basis of $X$ if and
only if every weaving is an unconditional basic sequence.   Surprisingly, we show that this is not always the case if some weavings are conditional, and in Example \ref{Ex:subspace} we give two unconditional Schauder bases for $\ell_1$ such that every weaving is a basic sequence, but there exists a weaving which is not a Schauder basis for all of $\ell_1$.

Lastly, in Section \ref{S:5} we prove two perturbation theorems for approximate Schauder frames and prove that these result in woven approximate Schauder frames.  For a thorough approach to the basics of frame theory, see \cite{CK, CL2, C}.  For background on Banach space theory see \cite{FHHMPZ} and \cite{LT}.

\section{Uniform constants for woven approximate Schauder
frames}\label{S:uniform}

One of the major results in \cite{BCGLL} is a proof that
if $(x^0_j)_{j=1}^\infty$ and $(x^1_j)_{j=1}^\infty$ are two  frames for a
Hilbert space $H$, then every weaving is a frame for $H$ if and only if there is
a uniform constant $C\geq 1$ such that $(x^0_j)_{j=1}^\infty$ and $(x^1_j)_{j=1}^\infty$ are $C$-woven.  We shall prove in Theorem \ref{T:weaklyWeaving} that this same uniformity result holds
more generally for approximate Schauder frames.

 Let
$(x^0_j,f^0_j)_{j=1}^\infty$ and $(x^1_j,f^1_j)_{j=1}^\infty$ be two
 approximate Schauder frames of $X$ such that every weaving is an approximate Schauder frame.
For
$I\subseteq\N$ and $\sigma\in \{0,1\}^\N$, we define a linear operator
$P_{\sigma,I}:X\rightarrow X$ by $P_{\sigma,I}(x)=\sum_{j\in
I}f^{\sigma(j)}_j(x)x_j^{\sigma(j)}$. Note that this series will
converge for every interval $I\subseteq\N$, but that the series may
not converge if the approximate Schauder frame is conditional and
$I$ is not an interval.  In this notation, we have that $P_{\sigma,\N}$ is the frame operator of
the weaving $(x^{\sigma(j)}_j,f^{\sigma(j)}_j)_{j=1}^\infty$.  Thus, to prove that two approximate Schauder frames
$(x^0_j,f^0_j)_{j=1}^\infty$ and $(x^1_j,f^1_j)_{j=1}^\infty$ are $C$-woven, we need to prove that
$\|P_{\sigma,\N}\|\leq C$ and $\|P_{\sigma,\N}^{-1}\|\leq C$ for all $\sigma\in\{0,1\}^\N$.

\begin{lem}\label{L:tail}
Let $(x^0_j,f^0_j)_{j=1}^\infty$ and $(x^1_j,f^1_j)_{j=1}^\infty$ be
two approximate Schauder frames for $X$ such that every weaving is an approximate Schauder frame for $X$. If $x\in
X$ and $\vp>0$, then there exists $N\in\N$ such that
$\|P_{\sigma,[m,n]}x\|<\vp$ for all $\sigma\in \{0,1\}^\N$ and $n\geq
m\geq N$.
\end{lem}

\begin{proof}
Let $x \in X$ and $\vp > 0$. Assume the contrary that for all $N\in\N$, we may
choose $\sigma\in\{0,1\}^\N$ and $n\geq m\geq N$ such that
$\|P_{\sigma,[m,n]}x\|\geq\vp$.  We may then inductively choose
$(\sigma_j)_{j=1}^\infty$ and $m_1\leq n_1<m_2\leq n_2<...$ such
that $\|P_{\sigma_j,[m_j,n_j]}x\|\geq\vp$.  As,
$([m_j,n_j])_{j=1}^\infty$ is a sequence of disjoint intervals, we
may choose $\sigma\in \{0,1\}^\N$ such that
$\sigma|_{[m_j,n_j]}=\sigma_j|_{[m_j,n_j]}$ for all $j\in\N$.  Thus,
we have that $\|P_{\sigma,[m_j,n_j]}x\|\geq\vp$ for all $j\in\N$. Since we assumed that $(x^{\sigma(j)}_j,f^{\sigma(j)}_j)$ is an approximate Schauder
frame, this contradicts the fact that $P_{\sigma,\N}x$ converges.
\end{proof}

\begin{lem}\label{L:upper}
Let $(x^0_j,f^0_j)_{j=1}^\infty$ and $(x^1_j,f^1_j)_{j=1}^\infty$ be
two approximate Schauder frames for $X$ such that every weaving is an approximate Schauder frame for $X$. There exists
a constant $D>0$ such that $\|P_{\sigma,[m,n]}\|\leq D$ for all
$\sigma\in\{0,1\}^\N$ and $m,n\in\N$.
\end{lem}
\begin{proof}
Let $x\in X$.  By Lemma \ref{L:tail} there exists $N\in\N$ such that
$\|P_{\sigma,[n,\infty)}x\|<1$ for all $\sigma\in\{0,1\}^\N$ and all
$n\geq N$. We have that $$\sup_{\sigma\in\{0,1\}^\N,n\leq
N}\|P_{\sigma,[n,\infty)}x\|\leq \sup_{\sigma\in\{0,1\}^\N,n\leq
N}\|P_{\sigma,[n,N)}x\|+ \sup_{\sigma\in\{0,1\}^\N}\|P_{\sigma,[N,\infty)}x\|.$$
Thus, we have
$$\sup_{\sigma\in\{0,1\}^\N,n\leq
N}\|P_{\sigma,[n,\infty)}x\|<\infty$$ 
as
$\sup_{\sigma}\|P_{\sigma,[N,\infty)}x\|\leq 1$ and the set
$\{P_{\sigma,[n,N)}x\}_{\sigma\in \{0,1\}^\N,n\leq N}$ is finite.
Hence, there exists a constant $K>0$ such that
$\|P_{\sigma,[n,\infty)}x\|\leq K$ for all $\sigma\in\{0,1\}^\N$ and
all $n\in\N$. Therefore, for all $m,n\in\N$ with $m\leq n$ we have
that
$$\|P_{\sigma,[m,n]}x\|=\|P_{\sigma,[m,\infty)}x-P_{\sigma,[n+1,\infty)}x\|\leq 2K.$$
Hence, $\{P_{\sigma,[m,n]}\}_{\sigma\in\{0,1\}^\N,m\leq n}$ is a set of bounded
operators such that for all $x\in X$, $\{P_{\sigma,[m,n]}x\}$ is
uniformly bounded in norm.  By the uniform boundedness
principle, the set of operators $\{P_{\sigma,[m,n]}\}$ is uniformly
bounded in norm.
\end{proof}

\begin{lem}\label{L:lower}
Let $(x^0_j,f^0_j)_{j=1}^\infty$ and $(x^1_j,f^1_j)_{j=1}^\infty$ be
two  approximate Schauder frames of $X$ such that every weaving is an approximate Schauder frame of $X$. For all $x\in
X$ there exists $\delta>0$ such that for all $\sigma\in\{0,1\}^\N$ we have
that $\|P_{\sigma,\N}x\|\geq\delta\|x\|$
\end{lem}

\begin{proof}

Let $x\in X$ such that $x\neq0$. By Lemma \ref{L:tail}, for all
$n\in\N$ there exists $M_n\in\N$ such that
$\|P_{\gamma,[M_n,\infty)}x\|<\frac{1}{n}$ for all
$\gamma\in\{0,1\}^\N$. Assume for the sake of contradiction that
for all $n\in\N$ there exists $\sigma_n\in\{0,1\}^\N$ such that
$\|P_{\sigma_n,\N}x\|<1/n$.  After passing to a subsequence, we may
assume that $(\sigma_n)_{n=1}^\infty$ converges to some
$\sigma\in\{0,1\}^\N$ and $\sigma_n|_{[1,M_n)}=\sigma|_{[1,M_n)}$
for all $n\in\N$.  For each $n\in\N$, we have the following
upperbound on the norm of $P_{\sigma,\N}x$.
\begin{align*}
\|P_{\sigma,\N}x\|&=\|P_{\sigma_n,\N}x-P_{\sigma_n,[M_n,\infty)}x+P_{\sigma,[M_n,\infty)}x\|\\
&\leq \|P_{\sigma_n,\N}x\|+\|P_{\sigma_n,[M_n,\infty)}x\|+\|P_{\sigma,[M_n,\infty)}x\|\\
&<1/n+1/n+1/n
\end{align*}
Thus, $\|P_{\sigma,\N}x\|<3/n$ for all $n\in\N$ and hence $P_{\sigma,\N}x=0$.  This contradicts that $(x^{\sigma(j)}_j,f^{\sigma(j)}_j)_{j=1}^\infty$ is an approximate Schauder frame.
\end{proof}

\begin{thm}\label{T:weaklyWeaving}
Given two approximate Schauder frames for a Banach space $X$, they
are woven if and only if every weaving is an approximate Schauder frame.  That is, if every weaving is an approximate Schauder frame then there is a uniform constant $C \geq 1$ such that every weaving is a $C$-approximate Schauder frame.  Furthermore, if every weaving is an unconditional approximate Schauder frame then there exists a constant $D\geq1$ such that every weaving is a $D$-unconditional approximate Schauder frame.
\end{thm}
\begin{proof}
Let $(x^0_j,f^0_j)_{j=1}^\infty$ and $(x^1_j,f^1_j)_{j=1}^\infty$ be
two approximate Schauder frames of $X$ such that every weaving is an approximate Schauder frame. Let $y\in X$.
By Lemma \ref{L:lower}, there exists $\delta>0$ such that
$\|y\|=\|P_{\sigma,\N}(P_{\sigma,\N}^{-1}y)\|\geq
\delta\|P_{\sigma,\N}^{-1}y\|$ for all $\sigma\in\{0,1\}^\N$.  Thus,
$\|P_{\sigma,\N}^{-1}y\|\leq (1/\delta)\|y\|$
 for all $\sigma\in\{0,1\}^\N$.  We have that $\{P^{-1}_{\sigma,\N}\}_{\sigma\in\{0,1\}^\N}$ is a collection of bounded operators such that $\sup_{\sigma\in\{0,1\}^\N}\|P^{-1}_{\sigma,\N}x\|<\infty$ for all $x\in X$.  Hence, $\sup_{\sigma\in\{0,1\}^\N}\|P^{-1}_{\sigma,\N}\|<\infty$ by the uniform
 boundedness principle.  Also,
 $\sup_{\sigma\in\{0,1\}^\N}\|P_{\sigma,\N}\|<\infty$ by Lemma
 \ref{L:upper}.  Thus, there is a uniform constant $C\geq 1$ such that $\|P_{\sigma,\N}\|\leq C$ and $\|P_{\sigma,\N}^{-1}\|\leq C$ for all $\sigma\in\{0,1\}^\N$ and hence
 $(x^0_j,f^0_j)_{j=1}^\infty$ and $(x^1_j,f^1_j)_{j=1}^\infty$ are $C$-woven approximate Schauder frames of
 $X$.

We now assume that every weaving of $(x^0_j,f^0_j)_{j=1}^\infty$ and $(x^1_j,f^1_j)_{j=1}^\infty$ is an unconditional approximate Schauder frame.
We have that there exists $C\geq 1$ such that every weaving of $(x^0_j,f^0_j)_{j=1}^\infty$ and $(x^1_j,f^1_j)_{j=1}^\infty$ is a $C$-approximate Schauder frame.
For the sake of contradiction we assume that for all $n\in\N$, there
exists $\sigma_n\in\{0,1\}^\N$ such that the approximate Schauder $(x^{\sigma_n(j)}_j,f^{\sigma_n(j)}_j)_{j=1}^\infty$ is not $n$-suppression unconditional.    After passing
to a subsequence, we may assume that there exists
$\sigma\in\{0,1\}^\N$ such that $\sigma_n\rightarrow\sigma$ in the
product topology of $\{0,1\}^\N$ and that $\sigma_n|_{[1,n]}=\sigma|_{[1,n]}$ for all $n\in\N$.
We choose a subsequence $(m_n)_{n=1}^\infty$ of $\N$ such that
there exists a finite set $\Gamma_n\subseteq[1,m_n]$ and $x_n\in X$ with $\|x_n\|=1$ so that $\|P_{\sigma_n,\Gamma_n}x_n\|\geq  n\|P_{\sigma_n,\N}x_n\|\geq nC^{-1}$.

We inductively define a sequence of natural numbers $(M_n)_{n=1}^\infty$ by $M_1=m_1$ and for $n\in\N$ we let $M_{n+1}=m_{M_n}$.
Let $\gamma\in\{0,1\}^\N$ be such that
$\gamma|_{[1,m_1]}=\sigma_1|_{[1,m_1]}$ and for $n\in\N$, $\gamma|_{(M_n,M_{n+1}]}=\sigma_{M_n}|_{(M_n,M_{n+1}]}$.  Let $D\geq1$ be the suppression unconditionality constant of  the weaving $(x^{\sigma(j)}_j,f^{\sigma(j)}_j)_{j=1}^\infty$.  For $n\in\N$, we have that
\begin{align*}
\|P_{\gamma,\Gamma_{M_n}\cap(M_n,M_{n+1}]}\|&=\|P_{\sigma_{M_n},\Gamma_{M_n}\cap(M_n,M_{n+1}]}\|\\
&\geq \|P_{\sigma_{M_n},\Gamma_{M_n}}\|-\|P_{\sigma_{M_n},\Gamma_{M_n}\cap[1,M_{n}]}\|\\
&=\|P_{\sigma_{M_n},\Gamma_{M_n}}\|-\|P_{\sigma,\Gamma_{M_n}\cap[1,M_{n}]}\|\geq M_n C^{-1}-DC.
\end{align*}
Thus, we have that $\sup_{n\in\N}\|P_{\gamma,\Gamma_{M_n}\cap(M_n,M_{n+1}]}\|=\infty$ and hence
$(x^{\gamma(j)}_j,f^{\gamma(j)}_j)_{j=1}^\infty$ is not
unconditional.

 \end{proof}

\section{Weaving bases}\label{S:3}

We now consider the case of weaving Schauder bases instead
of  approximate Schauder frames.
Recall that if $(x_j)_{j=1}^\infty$ is a Schauder basic sequence
then we say that it is {\em $C$-basic} for some $C\geq 1$ if  $\|\sum_{j=1}^N a_j x_j\|\leq
C\|\sum_{j=1}^\infty a_j x_j\|$ for all $N\in\N$ and all sequences
of scalars $(a_j)_{j=1}^\infty$ such that $\sum_{j=1}^\infty a_i
x_i$ converges.  We say that an unconditional Schauder basic sequence is {\em $C$-suppression unconditional} if
$\|\sum_{j\in E}a_j x_j\|\leq
C\|\sum_{j=1}^\infty a_j x_j\|$ for all $E\subseteq\N$ and all sequences
of scalars $(a_j)_{j=1}^\infty$ such that $\sum_{j=1}^\infty a_j
x_j$ converges.

In Section \ref{S:uniform} we proved that if every weaving of two
approximate Schauder frames is an approximate Schauder frame then
there is a uniform constant $C\geq 1$ such that the two approximate
Schauder frames are $C$-woven.  We now prove that this same
uniformity theorem holds for weaving Schauder bases.  The proof will follow the same concatonation argument as in the furthermore part of the proof of Theorem \ref{T:weaklyWeaving}.

\begin{prop}\label{P:basisConstant}
Let $(x_i^0)_{i=1}^\infty$ and $(x^1_i)_{i=1}^\infty$ be two
Schauder basic sequences in a Banach space $X$.  If
$(x^{\sigma(i)}_i)_{i=1}^\infty$ is a basic sequence for all
$\sigma\in\{0,1\}^\N$, then there is a uniform constant $C\geq1$ such
that for all $\sigma\in\{0,1\}^\N$, $(x^{\sigma(i)}_i)_{i=1}^\infty$
is  $C$-basic. Likewise,  if $(x^{\sigma(j)}_j)_{j=1}^\infty$ is an
unconditional basic sequence for all $\sigma\in\{0,1\}^\N$, then
there is a uniform constant $D\geq1$ such that for all
$\sigma\in\{0,1\}^\N$, $(x^{\sigma(j)}_j)_{j=1}^\infty$ is
$D$-unconditional.

\end{prop}

\begin{proof}
Assume by way of contradiction that  $(x^{\sigma(i)}_i)_{i=1}^\infty$ is a basic sequence
for all $\sigma\in\{0,1\}^\N$, but that for all $n\in\N$, there
exists $\sigma_n\in\{0,1\}^\N$ such that
$(x^{\sigma_n(i)}_i)_{i=1}^\infty$ is not $n$-basic.  After passing
to a subsequence, we may assume that there exists
$\sigma\in\{0,1\}^\N$ such that $\sigma_n\rightarrow\sigma$ in the
product topology of $\{0,1\}^\N$.  Let $C\geq1$ be such that
$(x^{\sigma(i)}_i)_{i=1}^\infty$ is $C$-basic. After passing to a
further subsequence, we may assume that
$\sigma_n|_{[1,n]}=\sigma|_{[1,n]}$ for all $n\in\N$.  Thus, for all
$n\in\N$, $(x_i^{\sigma_n(i)})_{i=1}^n$ is $C$-basic, and hence
$(x_i^{\sigma_n(i)})_{i=n+1}^\infty$ is not $(n-C)$-basic.  Let
$(m_n)_{n=1}^\infty$ be an increasing subsequence of natural numbers
such that for all $n\in\N$, $(x_i^{\sigma_n(i)})_{i=n+1}^{m_n}$ is
not $(n-C)$-basic.  We can now create $\gamma\in\{0,1\}^\N$ by
$\gamma|_{[1,m_1]}=\sigma_1|_{[1,m_1]}$,
$\gamma|_{(m_1,m_{m_1}]}=\sigma_{m_1}|_{(m_1,m_{m_1}]}$,
$\gamma|_{(m_{m_1},m_{m_{m_1}}]}=\sigma_{m_{m_1}}|_{(m_{m_1},m_{m_{m_1}}]}$,
and so on.  We have that $(x^{\gamma(i)}_i)_{i=1}^\infty$ is not
basic because it is formed by concatonating finite length basic
sequences whose basis constants go to infinity, which is a
contradiction.

A similar argument shows that if for all $\sigma\in\{0,1\}^\N$,
$(x^{\sigma(i)}_i)_{i=1}^\infty$ is an unconditional basic sequence
then there is a universal constant $D\geq1$ such that for all
$\sigma\in\{0,1\}^\N$,  $(x^{\sigma(i)}_i)_{i=1}^\infty$ is
$D$-unconditional.
\end{proof}

Suppose that $(x^0_j)_{j=1}^\infty$ and
$(x^1_j)_{j=1}^\infty$ are two Schauder bases for a Banach space $X$
with biorthogonal functionals $(x^{0*}_i)_{j=1}^\infty$ and
$(x^{1*}_i)_{j=1}^\infty$.  We can consider the problem of whether
or not $(x^0_i)_{j=1}^\infty$ and $(x^1_j)_{j=1}^\infty$ are woven
Schauder bases, and we can consider the problem of whether or not
$(x^0_j,x^{0*}_j)_{j=1}^\infty$ and $(x^1_j,x^{1*}_j)_{j=1}^\infty$
are woven approximate Schauder frames.  The first natural question
to ask is if these two notions are equivalent.  In other words, are
$(x^0_i)_{j=1}^\infty$ and $(x^1_j)_{j=1}^\infty$ woven Schauder
bases if and only if $(x^0_j,x^{0*}_j)_{j=1}^\infty$ and
$(x^1_j,x^{1*}_j)_{j=1}^\infty$ are woven approximate Schauder
frames?  Surprisingly, the answer is no.

\begin{ex}\label{E:conditional weaving} We have the following examples in $c_0$ and $\ell_1$.
\begin{enumerate}
\item Let $(e_j)_{j=1}^\infty$ be the standard unit vector basis for
$c_0$ with biorthogonal functionals $(e^*_j)_{j=1}^\infty$, and let
$(s_j)_{j=1}^\infty$ be the summing basis for $c_0$ with
biorthogonal functionals $(s^*_j)_{j=1}^\infty$.  That is,
$s_j=\sum_{i=1}^j e_i$ for all $j\in\N$. Then, $(s_j)_{j=1}^\infty$
is a conditional Schauder basis for $c_0$ which is woven with the
unconditional Schauder basis $(e_j)_{j=1}^\infty$, but
$(s_j,s^*_j)_{j=1}^\infty$ and $(e_j,e^*_j)_{j=1}^\infty$ are not
woven approximate Schauder frames.

\item Let $(e_j)_{j=1}^\infty$ be the standard unit vector basis for
$\ell_1$ with biorthogonal functionals $(e^*_j)_{j=1}^\infty$. Let
$x_{1}=e_{1}$ and $x_{n}=e_{n}-e_{n-1}$ for $n>1$.  Then,
$(x_j)_{j=1}^\infty$ is a conditional Schauder basis for $\ell_1$
which is woven with the unconditional Schauder basis
$(e_j)_{j=1}^\infty$, but $(x_j,x^*_j)_{j=1}^\infty$ and
$(e_j,e^*_j)_{j=1}^\infty$ are not woven approximate Schauder
frames, where $(x^*_j)_{j=1}^\infty$ is the sequence of biorthogonal
functionals to $(x_j)_{j=1}^\infty$.
\end{enumerate}
\end{ex}

Example \ref{E:conditional weaving} shows that it is possible for a
conditional Schauder basis to be woven with an unconditional
Schauder basis. The following proposition shows that this is
impossible for woven approximate Schauder frames.

\begin{prop}\label{P:Unconditional constants}
Let $(x^0_j,f^0_j)_{j=1}^\infty$ and $(x^1_j,f^1_j)_{j=1}^\infty$ be
$C$-approximate Schauder frames for a Banach space $X$.  Suppose
that $(x^0_j,f^0_j)_{j=1}^\infty$ is $K$-suppression unconditional
and that $(x^1_j,f^1_j)_{j=1}^\infty$ is not $D$-suppression
unconditional. Then $(x^0_j,f^{0}_j)_{j=1}^\infty$ and
$(x^1_j,f^{1}_j)_{j=1}^\infty$ are not $(DC^{-1}-KC)$-woven.
\end{prop}
\begin{proof}

Let $S_0$ be the frame operator for $(x^0_j,f^0_j)_{j=1}^\infty$ and
let $S_1$ be the frame operator for $(x^1_j,f^1_j)_{j=1}^\infty$.
Let $x\in X$ and $A\subset \N$ such that $\|\sum_{j\in A}
f^{1}_j(x)x^1_j\|>D\|S_1x\|\geq DC^{-1}\|x\|$.  Let $\sigma\in
\{0,1\}^\N$ be the characteristic function of $A$.  Then
\begin{align*}
\bigg\|\sum_{j\in\N} f_j^{\sigma(j)}(x) x_{j}^{\sigma(j)}\bigg\|
&\geq 
\bigg\|\sum_{j\in A} f^{1}_j(x)x^1_j\bigg\|-\bigg\|\sum_{j\not\in A}
f^{0}_j(x)x^0_j\bigg\|\\
&>D\|S_1x\|-K\|S_0x\|\\
&\geq
DC^{-1}\|x\|-KC\|x\|.
\end{align*}
Thus the frame operator of
$(x^{\sigma(j)}_j,f^{\sigma(j)}_j)_{j=1}^\infty$ has norm greater
than $(DC^{-1}-KC)$ and hence $(x^0_j,f^{0}_j)_{j=1}^\infty$ and
$(x^1_j,f^{1}_j)_{j=1}^\infty$ are not $(DC^{-1}-KC)$-woven.
\end{proof}

We conclude the section with two examples that again demonstrate the significance of conditional compared to unconditional convergence. The next construction shows that it is possible for
two woven unconditional Schauder bases to have some conditional
weavings.

\begin{ex}
Let $(e_j)_{j=1}^\infty$ be the standard unit vector basis for
$\ell_1$. Define $(x_i^0)_{i=1}^\infty$ by $x_{2n - 1}^0 = e_n$ and
$x^0_{2n} = e_{2n} - e_{2n - 1}$ for all $n \in \mathbb{N}$ and
define $(x_i^1)_{i=1}^\infty$ by $x_1^1 = e_1$, $x_2^1 = e_2$, and
$x_{2n - 1}^1 = e_{2n-1} - e_{2n - 2}$ and $x_{2n}^1 = e_{2n}$ for
all $n \geq 2$. Then $(x_i^0)_{i=1}^\infty$ and
$(x_i^1)_{i=1}^\infty$ are woven Schauder bases, each of which are
unconditional.  However, the alternating weaving given by $(x^1_1,x^0_2,x^1_3,x^0_4,...)=(e_1,e_2-e_1,e_3-e_2,e_4-e_3,...)$ is the same conditional basis used in Example \ref{E:conditional weaving}.
\end{ex}

The following construction shows it is possible for all the weavings
of two unconditional Schauder bases to be Schauder basic sequences,
but there exists a weaving which does not have dense span.  In \cite{BCGLL}, it is proven that this is impossible for Riesz bases for Hilbert spaces, and in Theorem \ref{T:unc} we prove that this is impossible if all weavings are unconditional.

\begin{ex}\label{Ex:subspace}
Let $(e_j)_{j=1}^\infty$ be the standard unit vector basis for $\ell_1$. Define $(x_i^0)_{i = 1}^\infty$ by $x_{2n-1}^0 = e_{2n-1} + e_{2n}$ and $x_{2n}^0 = e_{2n-1} - e_{2n}$ for all $n \in \mathbb{N}$ and define $(x_i^1)_{i = 1}^\infty$ by $x_1^1 = e_1$ and $x_{2n}^1 = e_{2n} + e_{2n+1}$ and $x_{2n+1}^1 = e_{2n} - e_{2n+1}$ for all $n \in \mathbb{N}$. Then $(x_i^0)_{i=1}^{\infty}$
and $(x_i^1)_{i=1}^{\infty}$ are both unconditional Schauder bases of $\ell_1$ and every weaving  is a Schauder basic sequence.  However, the alternating weaving
$(x^0_1,x^1_2,x^0_3,x^1_4,...)=(e_1+e_2,e_2+e_3,e_3+e_4,e_4+e_5,...)$ does not contain $e_1$ in its closed span and thus is not a basis.
\end{ex}

\section{Weaving unconditional Schauder bases}

In this section, we consider weaving unconditional bases. First, we show that two woven unconditional bases are necessarily equivalent. We conclude the section by proving that two unconditional bases that are woven and have all unconditional weavings is equivalent to five other characterizations, including in particular that all weavings are unconditional basic sequences and all weavings of the corresponding approximate Schauder frames are unconditional approximate Schauder frames.  Recall that in Section \ref{S:3} we gave counterexamples to these equivalences when it is not the case that all weavings are unconditional.
 We first cite the following theorem, which can be found in \cite{LT} along with a thorough introduction to Banach lattices.

\begin{thm}\label{square}
Let $(x_i)_{i=1}^\infty$ be an unconditional basis for a Banach lattice $X$.  There exists a constant $D\geq1$ such that for every $n\in\N$ and sequence of scalars $(a_i)_{i=1}^n$, we have that
$$D^{-1}\left\Vert\left(\sum_{i=1}^n |a_i x_i|^2\right)^{1/2}\right\Vert\leq \left\Vert \sum_{i=1}^n a_i x_i\right\Vert\leq D \left\Vert\left(\sum_{i=1}^n |a_i x_i|^2\right)^{1/2}\right\Vert.
$$
\end{thm}

When we apply Theorem \ref{square}, the Banach lattice structure for $X$ will be defined by a different 1-unconditional basis $(u_i)_{i=1}^\infty$ with biorthogonal functionals $(u^*_i)_{i=1}^\infty$.  In this case, if $(x_i)_{i=1}^\infty$ is a sequence in $X$, and $(a_i)_{i=1}^n$ is a finite sequence of scalars then the vector
$\left(\sum_{i=1}^n |a_i x_i|^2\right)^{1/2}$ is given by
\begin{equation}\label{E:lattice}
\left(\sum_{i=1}^n |a_i x_i|^2\right)^{1/2}=\sum_{j=1}^\infty \left(\sum_{i=1}^n|a_i u^*_j(x_i)|^2\right)^{1/2}\!\!\!u_j.
\end{equation}

We have previously given an example of a conditional Schauder basis for $\ell_1$ such that every weaving with the unit vector basis for $\ell_1$ is a Schauder basis.  In particular, this gives an example of two non-equivalent Schauder bases for a Banach space $X$ such that every weaving is a Schauder basis for $X$.  Our next result shows that this is impossible if both bases are unconditional.  Recall that two Schauder basic sequences $(x_j)_{j=1}^\infty$ and $(y_j)_{j=1}^\infty$ are said to be {\em equivalent} if there are constants $c,C>0$ such that $c\|\sum a_j x_j\|\leq \|\sum a_j y_j\|\leq C\|\sum a_j x_j\|$ for all $(a_j)_{j=1}^\infty\in \coo$.

\begin{thm}\label{T:equivalent}
Let $(x^0_i)_{i=1}^\infty$ and $(x^1_i)_{i=1}^\infty$ be two  semi-normalized unconditional bases for a Banach space $X$ such that for all $\sigma\in\{0,1\}^\N$, $(x^{\sigma(i)}_i)_{i=1}^\infty$ is a basic sequence.  Then,  $(x^0_i)_{i=1}^\infty$ and $(x^1_i)_{i=1}^\infty$ are equivalent.
\end{thm}

\begin{proof}
After renorming, we may assume that $(x^1_i)_{i=1}^\infty$ is
1-unconditional.  Let $X$ have the Banach lattice structure defined
by the 1-unconditional basis $(x^1_i)_{i=1}^\infty$.  Let $D$ be the
constant given in Theorem \ref{square}.  As every weaving of
$(x^0_i)_{i=1}^\infty$ and $(x^1_i)_{i=1}^\infty$ is
a basic sequence, there is
a constant $C>0$ such that $|x^{1*}_i(x^0_i)|\geq C$ for all
$i\in\N$. For any $(a_i)_{i=1}^\infty\in\coo$, we have the following
estimate:
\begin{align*}
\left\|\sum_{i=1}^\infty a_i x^0_i\right\|&\geq D^{-1}\left\|\left(\sum_{i=1}^\infty |a_i x^0_i|^2\right)^{1/2}\right\|\quad\textrm{ by Theorem }\ref{square},\\
&=D^{-1}\left\|\sum_{j=1}^\infty \left(\sum_{i=1}^\infty |a_i x^{1*}_j(x^0_i)|^2\right)^{1/2}x^1_j\right\|\quad\textrm{ by }\eqref{E:lattice},\\
&\geq CD^{-1}\left\|\sum_{j=1}^\infty a_j x^1_j\right\|\quad
\textrm{ as }(x^1_j)\textrm{ is }1\textrm{-unconditional and
}|x^{1*}_j(x^0_j)|\geq C\textrm{ for all }j\in\N.
\end{align*}
Therefore, we have that the basis $(x^0_i)_{i=1}^\infty$ dominates the basis $(x^1)_{i=1}^\infty$, and by the same argument $(x^1_i)_{i=1}^\infty$ dominates $(x^0)_{i=1}^\infty$.  Hence, $(x^0_i)_{i=1}^\infty$ and $(x^1)_{i=1}^\infty$ are equivalent.
\end{proof}

Note that in proving Theorem \ref{T:equivalent} it was only required that there exists $C>0$ such that $|x^{0*}_j(x^1_j)|\geq C$ and $|x^{1*}_j(x^0_j)|\geq C$ for $j\in\N$ which is a much weaker property than
$(x^0_i)_{i=1}^\infty$ and $(x^1_i)_{i=1}^\infty$ are woven.

The following lemmas will be used in characterizing when every weaving of two unconditional Schauder bases is also an unconditional Schauder basis.

\begin{lem}\label{babyproj}
Suppose $X$ is a Banach space having subspaces $Y$ and $Z$. Let $P:X \to Y$ be a projection from $X$ onto $Y$ and assume $P|_Z:Z \to Y$ is invertible. Then $Q = (P|_Z)^{-1} P$ is a projection onto $Z$.
\end{lem}

\begin{proof}
The fact that $Q$ is onto $Z$ is clear. We check $Q$ is a projection:
\begin{align*}
Q^2 = (P|_Z)^{-1} \big(P (P|_Z)^{-1} P\big) = (P|_Z)^{-1} P = Q
\end{align*}
since if $x \in X$, then $Px \in Y$ and so $P (P|_Z)^{-1} Px = Px$.
\end{proof}

For a Banach space $X$ with subspaces $X_0,X_1\subseteq X$, the distance between $X_0$ and $X_1$ measures how close a point in the sphere of one space can be to a point in the other space. That is,
$$d(X_0,X_1)=\inf\{\|x-y\|\,:\, x\in S_{X_j},y\in X_{1-j}\textrm{ and }j\in\{0,1\}.\}
$$

\begin{lem}\label{projections}
Suppose $X$ is a Banach space and $X = X_1 \oplus X_2 = Y_1 \oplus Y_2$ for some subspaces $X_1$, $X_2$, $Y_1$, $Y_2$ of $X$. Assume $P$ is a projection onto $X_1$ with $(I-P)$ a projection onto $X_2$, and $Q$ is a projection on $Y_1$ with $(I-Q)$ projecting onto $Y_2$. Assume $Q|_{X_1}:X_1 \to Y_1$ and $(I-P)|_{Y_2}:Y_2 \to X_2$ are invertible and $d(X_1,Y_2) > 0$. Then $X = X_1 \oplus Y_2$.
\end{lem}

\begin{proof}
By Lemma \ref{babyproj}, we have that $(Q|_{X_1})^{-1} Q$ is a projection onto $X_1$ and
$\big( (I-P)|_{Y_2} \big)^{-1} (I-P)$ is a projection onto $Y_2$.  Furthermore, both
\[
(Q|_{X_1})^{-1} Q \big( (I-P)|_{Y_2} \big)^{-1} (I-P) = 0 \quad \mbox{and} \quad \big( (I-P)|_{Y_2} \big)^{-1} (I-P)(Q|_{X_1})^{-1} Q = 0.
\]
Thus, the operator $R:X \to X$ defined by
\[
R := (Q|_{X_1})^{-1} Q + \big( (I-P)|_{Y_2} \big)^{-1} (I-P)
\]
is a projection from $X$ onto $X_1 \oplus Y_2$. We will show that the projection $R$ is one-to-one (and hence the identity map) to prove that $X = X_1\oplus Y_2$.
Let $x \in X$ be so that $Rx = 0$. Then
\[\|Rx\| = \|(Q|_{X_1})^{-1} Q x+ \big( (I-P)|_{Y_2} \big)^{-1} (I-P)x\|=0
\]
and since $d(X_1,Y_2) > 0$, it follows that
\[
(Q|_{X_1})^{-1} Q x  =  0 \quad \mbox{and} \quad \big( (I-P)|_{Y_2} \big)^{-1} (I-P)x = 0.
\]
Thus, $Qx = 0 = (I-P)x$ so that $x = Px = (I-Q)x$. However, since $(Q|_{X_1})^{-1} Q$ is a projection onto $X_1$, we have $Px = (Q|_{X_1})^{-1} Q Px = (Q|_{X_1})^{-1} Q x = 0$ and therefore $x = 0$.
\end{proof}

Given a sequence $(x_i)_{i\in\N}$ in a Banach space $X$ and a set $\Gamma\subseteq\N$, we let $[x_i]_{i\in\Gamma}$ denote the subspace $\overline{\spa}\{x_i\}_{i\in\Gamma} $.

\begin{thm}\label{T:unc}
Let $(x_i^0)_{i = 1}^\infty$ and $(x_i^1)_{i = 1}^\infty$ be two $C$-unconditional bases for a Banach space $X$ with biorthogonal functionals $(x_i^{0*})_{i = 1}^\infty$ and $(x_i^{1*})_{i = 1}^\infty$, respectively. The following are equivalent:
\begin{enumerate}[(i)]
\item For every $\sigma \in \{0,1\}^\N$, $(x_i^{\sigma(i)})_{i = 1}^\infty$ is an unconditional basis for $X$.
\item There is a $K > 0$ so that for any $\sigma \in \{0,1\}^\N$, $(x_i^{\sigma(i)})_{i = 1}^\infty$ is a $K$-unconditional basis for $X$.

\item For every $\sigma \in  \{0,1\}^\mathbb{N}$, $(x_i^{\sigma(i)}, x_i^{\sigma(i)*})_{i = 1}^\infty$ is an unconditional approximate Schauder frame.

\item For every $\sigma \in \{0,1\}^\mathbb{N}$, $(x_i^{\sigma(i)})_{i = 1}^\infty$ is an unconditional basic sequence.

\item There exists a $D > 0$ so that for every $\sigma \in \{0,1\}^\mathbb{N}$,
\[
d\left([x_i^0]_{i \in \sigma^{-1}(0)}, [x_i^1]_{i \in \sigma^{-1}(1)}\right) \geq D^{-1}.
\]

\item There exists $E>0$ such that for any $\sigma \in  \{0,1\}^\mathbb{N}$, if $P$ is the basis projection of $(x^0_i)_{i=1}^\infty$ onto $[x_i^0]_{i \in \sigma^{-1}(0)}$ and $Q$ is the basis projection of $(x^1_i)_{i=1}^\infty$ onto $[x_i^1]_{i \in \sigma^{-1}(0)}$  then $$P|_{[x_i^1]_{i \in \sigma^{-1}(0)}}:{[x_i^1]_{i \in \sigma^{-1}(0)}}\!\!\rightarrow {[x_i^0]_{i \in \sigma^{-1}(0)}}\quad\textrm{and}\quad Q|_{[x_i^0]_{i \in \sigma^{-1}(0)}}:{[x_i^0]_{i \in \sigma^{-1}(0)}}\!\!\rightarrow {[x_i^1]_{i \in \sigma^{-1}(0)}}$$
are invertible with inverses having norm at most $E$.


\end{enumerate}
\end{thm}

\begin{proof}

Note that (i)$\Rightarrow$(ii) follows from Proposition \ref{P:basisConstant}, and that (ii)$\Rightarrow$(iv) is immediate.  We will prove
(iv)$\Rightarrow$(v)$\Rightarrow$(vi)$\Rightarrow$(i) and
(iii)$\Rightarrow$(vi)$\Rightarrow$(iii).

The case (iv)$\Rightarrow$(v) is well known, but we include a short proof here for completeness.  By Proposition \ref{P:basisConstant} there exists $K>0$ such that $(x^{\sigma(i)}_i)_{i\in\N}$ is $K$-suppression unconditional for all $\sigma\in\{0,1\}^\N$.  Let $\sigma\in\{0,1\}^\N$, $x\in[x^0_i]_{i\in\sigma^{-1}(0)}$, and $y\in[x^1_i]_{i\in\sigma^{-1}(1)}$.  We have that $\|x-y\|\geq\frac{1}{K}\|x\|$, and hence $d\big([x_i^0]_{i \in \sigma^{-1}(0)}, [x_i^1]_{i \in \sigma^{-1}(1)}\big) \geq \frac{1}{K}$.

For (v)$\Rightarrow$(vi),  we assume that (v) holds for some $D>0$.
Let $\sigma\in\{0,1\}^\N$.  For $\Gamma\subseteq\N$ we let
$P_\Gamma$ be the basis projection of $(x^0_i)_{i=1}^\infty$ onto
$[x_i^0]_{i \in \Gamma}$.   Note that this implies that $I-P_{\Gamma}=P_{\Gamma^c}$ is the basis projection  of $(x^0_i)_{i=1}^\infty$ onto
$[x_i^0]_{i \in \Gamma^c}$.  Let $x\in[x_i^1]_{i \in \sigma^{-1}(0)}$
with $\|x\|=1$. Then
\begin{equation}\label{E:Projectionx}
\left\|P_{\sigma^{-1}(0)}(x)\right\|=\left\|x-(x-P_{ \sigma^{-1}(0)}(x))\right\|\geq
d\left([x_i^1]_{i \in \sigma^{-1}(0)}, [x_i^0]_{i \in \sigma^{-1}(1)}\right)
\geq D^{-1}.
\end{equation}
Thus, we have that $P_{\sigma^{-1}(0)}$ is an isomorphism of
$[x_i^1]_{i \in \sigma^{-1}(0)}$ onto a closed subspace of
$[x_i^0]_{\sigma^{-1}(0)}$ with inverse having norm at most $E := D$.  It remains to prove that
$P_{\sigma^{-1}(0)}\left([x_i^1]_{\sigma^{-1}(0)}\right)$ is dense in
$[x_i^0]_{i \in \sigma^{-1}(0)}$.

Let $y\in \spa(x_i^0)_{i \in \sigma^{-1}(0)}$
with $\|y\|=1$.  As
$(x_i^0)_{i=1}^\infty$ and $(x^1_i)_{i=1}^\infty$ are unconditional
bases there exists $N\in\N$ such that $y\in\spa_{1\leq i\leq
N}x^0_i$
and for all $\Gamma\subset(N,\infty)$ we have that
$\big\|\sum_{i\in \Gamma}x^{1*}_i(y)x^1_i\big\|<\vp$. For $n\geq N$, we let
$\sigma_n\in\{0,1\}^\N$ be defined by
$\sigma_n|_{[1,n)}=\sigma|_{[1,n)}$ and $\sigma_n(i)=1$ for $i\geq
n$.  A simple induction argument gives that
$(x^{\sigma_n(i)}_i)_{i=1}^\infty$ is an unconditional basis for $X$
as we have changed only finitely many of the coordinates of
$(x^1_i)_{i=1}^\infty$  and $d\big([x_i^0]_{i \in \sigma_n^{-1}(0)},
[x_i^1]_{i \in \sigma_n^{-1}(1)}\big) >0$.  Furthermore,
$P_{\sigma_n^{-1}(0)}$ is an  isomorphism of
$[x^1_i]_{\sigma_n^{-1}(0)}$ onto $[x^0_i]_{\sigma_n^{-1}(0)}$ because these spaces are finite dimensional and have the same dimension.
Hence, for each $n\geq N$ there exists unique
$z_n\in \spa(x^1_i)_{i\in\sigma_n^{-1}(0)}$ such that $y=P_{\sigma_n^{-1}(0)}(z_n)=\sum_{i\in\sigma_n^{-1}(0)}x^{0*}_i(z_n) x^0_i$ and $\|z_n\|\leq D$ by \eqref{E:Projectionx}.
Therefore,
\begin{equation}\label{E:tail}
P_{\sigma^{-1}(0)}z_n-y=\sum_{i\in\sigma^{-1}(0)}
x^{0*}_i(z_n)x^{0}_i- \sum_{i\in\sigma_n^{-1}(0)}x^{0*}_i(z_n) x^0_i=\sum_{i\in\sigma^{-1}(0)\cap(n,\infty)}x^{0*}_i(z_n)x^{0}_i.
\end{equation}

If $(P_{\sigma^{-1}(0)}z_n)_{n=1}^\infty$ has a sequence of convex
combinations which converges to $y$ then we can conclude that $y\in
P_{\sigma^{-1}(0)}([x^1_i]_{i\in\sigma^{-1}(0)})$ as the set is
closed and convex. We assume by contradiction that this is not the case. Since any convex combination of $y$ is itself, there exists $F>0$ such that for all $n \geq N$,
\begin{equation}\label{E:big tail}
\left\|\sum_{i\in\sigma^{-1}(0)\cap(n,\infty)}\sum_{j=n}^\infty
x^{0*}_i(a_j z_j)x^{0}_i\right\|>2F \textrm{ for all
}(a_j)_{j=n}^\infty\in [0,1]^\N\textrm{ with }\sum_{j=
n}^\infty a_j=1.
\end{equation}

Let  $\vp_n\searrow0$ such that $\sum_{n=1}^\infty n\vp_n<1$ and $\sum_{n=1}^\infty \vp_n<F$.  We inductively choose a subsequence $(M_n)_{n=-1}^\infty$ of
$\N$ with $M_{-1}=1$ and $M_0=N+1$ and sequences $(y_n)_{n=1}^\infty,(w_n)_{n=0}^\infty\subseteq X$ such that
\begin{enumerate}
\item[(a)]  $\big\|\sum_{i\in\sigma^{-1}(0)\cap[M_{n+1},\infty)}x^{0*}_i(z_{M_j})x^0_i\big\|<\vp_{n+1}$ for all $1\leq j\leq n$ and $n\in \N$.
\item[(b)] $y_n:=\sum_{i\in\sigma^{-1}(0)\cap[M_{n},M_{n+1})}x^{0*}_i(z_{M_n})x^0_i$ for all $n \in \mathbb{N}$ and $\big\|\sum _{j\in\N} a_ j y_j\big\|>F$ for all $(a_j)_{j\in\N}\in[0,1]^\mathbb{N}$ with $\sum_{j\in\N} a_j=1$.

\item[(c)]$w_n:=\lim\limits_{j\rightarrow\infty}\sum_{i\in\sigma^{-1}(0)\cap[M_{n-1},M_n)}x^{1*}_i(z_{M_j})x^1_i$ for all $n\geq0$.
\item [(d)]
$\big\|w_n-\sum_{i\in\sigma^{-1}(0)\cap[M_{n-1},M_{n})}x^{1*}_i(z_{M_j})x^1_i\big\|<\vp_{j}/j$ for all $j> n\geq0$.
\end{enumerate}
Note that (a) can be achieved because  the series converges and we
are only working with finitely many $z_j$ at each step. We have that
(b) follows from \eqref{E:big tail}, (a), and
$\sum_{n=1}^\infty\vp_n<F$.  The definition of $w_n$ in (c) can be
made because the sum is over the finite interval $[M_{n-1}, M_n)$
and the sequence $(z_{M_j})_{j =n+1}^\infty$ is bounded, so we may
pass to a subsequence of $(M_j)_{j \in \mathbb{N}}$ in which the
limit exists. Finally, since the limit exists by (c), we can pass to
yet another subsequence of $(M_j)_{j \in \mathbb{N}}$ to obtain (d).

For all $n\in\N$, we set $v_n:=\sum_{i\in\sigma^{-1}(0)\cap
[M_{n-1},M_n)}x^{1*}_i(z_{M_n})x^1_i$ and obtain the following:
\begin{enumerate}
\item[(e)] $\big\|z_{M_n}-(v_n+\sum_{j=0}^{n-1} w_j)\big\|< \vp_n$ for all $n\in\N$. 

\item[(f)] $\big\|P_{\sigma^{-1}(0)}z_{M_n}-(y+y_n)\big\|<\vp_{n+1}$ for all $n \in \mathbb{N}$. 
\end{enumerate}

We have that (e) follows from (d) and the fact that $z_{M_n}\in[x_j^1]_{j=1}^{M_n-1}$.  We have that (f) follows from (a), (b), and \eqref{E:Projectionx}.

Let $M\in\N$ and $\Omega= \sigma^{-1}(0)\cap \cup_{n=1}^{2M}[M_{2n-1},M_{2n})$.
Then
\begin{equation}\label{E:support}
\frac{1}{2M}\sum_{n=1}^{2M} (-1)^n v_{2n}\in [x^1_i]_{i\in\Omega} \quad \mbox{and} \quad \frac{1}{2M}\sum_{n=1}^{2M} (-1)^n ( P_{\sigma^{-1}(1) }z_{M_{2n}}+ y_{2n})\in [x^0_i]_{i\in\N\setminus\Omega}.
\end{equation}
Since $(y_n)_{n=N}^\infty$ is a block sequence of $(x^0_i)_{i\in\sigma^{-1}(0)}$, it is also $C$-unconditional and therefore by (b) we have that
\begin{equation}\label{E:lowerAvg}
\left\|\dfrac{1}{2M}\sum_{n=1}^{2M} (-1)^n \left( P_{\sigma^{-1}(1)} z_{M_{2n}}+  y_{2n}\right)\right\|
\geq C^{-1}\left\|\dfrac{1}{2M}\sum_{n=1}^{2M} y_{2n}\right\|
>
 C^{-1}F.
\end{equation}
As $(w_i)_{ i=0}^\infty$ is a block sequence of $(x^1_i)_{i=1}^\infty$, we have
 for all $n\in\N$ and $\Gamma\subseteq\{0,1,...,n\}$ that
\begin{equation}\label{E:sumwi}
\left\|\sum_{i\in\Gamma} w_i\right\|\leq C \left\|\sum_{i=1}^n w_i\right\|=C\left\|\lim_{j\rightarrow\infty}\sum_{i\in\sigma^{-1}(0)\cap[1,M_n)}x^{1*}_i(z_{M_j})x^1_i\right\|\leq \limsup_{j\rightarrow\infty}C^2\|z_{M_j}\|\leq C^2D.
\end{equation}
Thus, we have the following estimate:
\begin{align*}
&\left\|\frac{1}{2M}\sum_{n=1}^{2M}\right.  (-1)^n \left. (P_{\sigma^{-1}(1)} z_{M_{2n}}+y_{2n})-\frac{1}{2M}\sum_{n=1}^{2M} (-1)^n v_{2n}\right\|\\ &=
\frac{1}{2M}\left\|\sum_{n=1}^{2M} (-1)^n (P_{\sigma^{-1}(1)} z_{M_{2n}}+y_{2n}-v_{2n})\right\|\\
&\leq \frac{1}{2M}\left\|\sum_{n=1}^{2M} (-1)^n (y+y_{2n}+P_{\sigma^{-1}(1)} z_{M_{2n}}- (v_{2n}+\sum_{j=0}^{2n-1} w_j) )\right\|+\frac{1}{2M} \left\|\sum_{j=1}^{M}w_{4j-2}+w_{4j-1}\right\| \\
&\leq  \frac{1}{2M}\sum_{n=1}^{2M}  \left\|y+y_{2n}+P_{\sigma^{-1}(1)} z_{M_{2n}}- (v_{2n}+\sum_{j=0}^{2n-1} w_j )\right\|+\frac{C^2 D }{2M}\quad\textrm{ by }\eqref{E:sumwi} \\
&< \frac{1}{2M}\sum_{n=1}^{2M}  \left\|P_{\sigma^{-1}(0)}z_{M_{2n}}\!+\!P_{\sigma^{-1}(1)} z_{M_{2n}}\!-\! z_{M_{2n}}\right\|+\frac{1}{2M}\sum_{n=1}^{2M}(\vp_{2n}+\vp_{2n+1})\!+\!\frac{C^2 D }{2M}\textrm{ by (e) and (f)} \\
&< \frac{1}{2M}+\frac{C^2 D }{2M}.
\end{align*}
The above estimate together with \eqref{E:support} and \eqref{E:lowerAvg} gives that
$d([x^0_i]_{i\in\N\setminus\Omega},[x^1_i]_{i\in\Omega})<CF^{-1}(1+C^2 D)(2M)^{-1}$.  This contradicts (v) by choosing $M$ large enough. An identical proof shows the claim for $Q$.

For (vi)$\Rightarrow$(i), we assume that (vi) holds for some $E>0$.
Let $\sigma\in\{0,1\}^\N$.  For $j\in\sigma^{-1}(0)$ we let
$f_j:=\big(Q|_{[x^0_i]_{i\in\sigma^{-1}(0)}}^{-1}Q\big)^{*}x^{0*}_j$ and for
$j\in\sigma^{-1}(1)$ we let
$f_j:=\big((I-P)|_{[x^1_i]_{i\in\sigma^{-1}(1)}}^{-1}(I-P)\big)^{*}x^{1*}_j$.
We first prove that $(x^{\sigma(i)}_i,f_i)_{i=1}^\infty$ is a
biorthogonal system.  Let $j,k\in\sigma^{-1}(0)$, then
$$f_j(x^0_k)=\left(\big(Q|_{[x^0_i]_{i\in\sigma^{-1}(0)}}^{-1}Q\big)^{*}x^{0*}_j\right)(x_k^0)=x^{0*}_j\left(\big(Q|_{[x^0_i]_{i\in\sigma^{-1}(0)}}^{-1}Q\big)x^0_k\right)=x^{0*}_j(x^0_k)=\delta_{j,k}$$
For $j\in \sigma^{-1}(0)$ and $k\in\sigma^{-1}(1)$,
$$f_j(x^1_k)=\left(\big(Q|_{[x^0_i]_{i\in\sigma^{-1}(0)}}^{-1}Q\big)^{*}x^{0*}_j\right)(x_k^1)=x^{0*}_j\left(\big(Q|_{[x^0_i]_{i\in\sigma^{-1}(0)}}^{-1}Q\big)x^1_k\right)=x^{0*}_j(0)=0$$
Thus, if $j\in\sigma^{-1}(0)$ and $k\in\N$, then
$f_j(x^{\sigma(k)}_k)=\delta_{j,k}$. Likewise, if
$j\in\sigma^{-1}(1)$ and $k\in\N$, then
$f_j(x^{\sigma(k)}_k)=\delta_{j,k}$.  Thus, $(f_i)_{i=1}^\infty$ are
biorthogonal functionals to $(x_i^{\sigma(i)})_{i=1}^\infty$.  Let
$\Gamma\subseteq\N$ and $x\in X$.
\begin{align*}
\left\|\sum_{i\in\Gamma}f_i(x)x^{\sigma(i)}_i\right\|&\leq
\left\|\sum_{i\in\Gamma\cap\sigma^{-1}(0)}\!\!\!\!\!\!(\!Q|_{[x^0_i]_{i\in\sigma^{-1}(0)}}^{-1}\!\!\!\!\!\!Q)^*x^{0*}_i(x)x^{0}_i\right\|
\!+\!
\left\|\sum_{i\in\Gamma\cap\sigma^{-1}(1)}\!\!\!\!\!\!(\!(I-P)|_{[x^0_i]_{i\in\sigma^{-1}(0)}}^{-1}\!\!\!\!\!\!\!\!(I-P))^*x^{1*}_i(x)x^{1}_i\right\|\\
&
\leq C\left\|\sum_{i\in\N}x^{0*}_i(Q|_{[x^0_i]_{i\in\sigma^{-1}(0)}}^{-1}\!\!\!\!Q x)x^{0}_i\right\|
+
C\left\|\sum_{i\in\N}x^{1*}_i((I-P)|_{[x^0_i]_{i\in\sigma^{-1}(0)}}^{-1}\!\!\!\!\!\!(I-P)x)x^{1}_i\right\|\\
&
= C\left\|Q|_{[x^0_i]_{i\in\sigma^{-1}(0)}}^{-1}\!\!Q x \right\|
+
C\left\|(I-P)|_{[x^0_i]_{i\in\sigma^{-1}(0)}}^{-1}\!\!\!\!\!\!(I-P) x\right\|\\
&\leq C\left\|Q|_{[x^0_i]_{i\in\sigma^{-1}(0)}}^{-1}\!\!Q \right\|\|x\|
+
C\left\|(I-P)|_{[x^0_i]_{i\in\sigma^{-1}(0)}}^{-1}\!\!\!\!\!\!(I-P) \right\|\|x\|\\
&\leq 2CE\|x\|
\end{align*}
Thus, $(x^{\sigma(i)})_{i=1}^\infty$ is an unconditional basic
sequence.  By Lemma \ref{projections}, we have that
$[x^{\sigma(i)}_i]_{i\in\N}=X$ and hence
$(x^{\sigma(i)})_{i=1}^\infty$ is an unconditional basis of $X$.

For (vi)$\Rightarrow$(iii), we assume that (vi) holds for some $E>0$.  Let $\sigma\in\{0,1\}^\N$.  We let   $P$ be the basis projection of $(x^0_i)_{i=1}^\infty$ onto $X_1:=[x_i^0]_{i \in \sigma^{-1}(0)}$ and $Q$ be the basis projection of $(x^1_i)_{i=1}^\infty$ onto $Y_1:=[x_i^1]_{i \in \sigma^{-1}(0)}$.  Note that $(I-P)$ is the
basis projection of $(x^0_i)_{i=1}^\infty$ onto $X_2:=[x_i^0]_{i \in \sigma^{-1}(1)}$ and $(I-Q)$ is the basis projection of $(x^1_i)_{i=1}^\infty$ onto $Y_2:=[x_i^1]_{i \in \sigma^{-1}(1)}$.  We have that $P|_{Y_1}$, $Q|_{X_1}$, $(I-P)|_{Y_2}$, and $(I-Q)|_{Y_1}$ are all invertible with bounded inverses.  We let $S$ be the frame operator of $(x^{\sigma(i)}_i,x^{\sigma(i)*}_i)_{i=1}^\infty$.  Then for all $x\in X$, we have that
$$
S(x)=\sum_{i=1}^\infty x^{\sigma(i)*}_i(x)x^{\sigma(i)}_i=\sum_{i\in\sigma^{-1}(0)}x^{0}_i(x)x^{0}_i
+\sum_{i\in\sigma^{-1}(1)}x^{1}_i(x)x^{1}_i=P(x)+(I-Q)(x).
$$
Thus, $S=P+(I-Q)$ is the frame operator of
$(x^{\sigma(i)}_i,x^{\sigma(i)*}_i)_{i=1}^\infty$, and hence is well
defined and bounded.  We consider the operator
$T=P|^{-1}_{Y_1}Q|^{-1}_{X_1}Q+(I-Q)|^{-1}_{X_2}(I-P)|^{-1}_{Y_2}(I-P)$
and will prove that $T$ is the inverse of $S$ and hence
$(x^{\sigma(i)}_i,x^{\sigma(i)*}_i)_{i=1}^\infty$ would be an
approximate Schauder frame.  We have that
\begin{align*}
ST&=(P+(I-Q))P|^{-1}_{Y_1}Q|^{-1}_{X1}Q+(P+(I-Q))(I-Q)|^{-1}_{X_2}(I-P)|^{-1}_{Y_2}(I-P)\\
&=PP|^{-1}_{Y_1}Q|^{-1}_{X1}Q+(I-Q)(I-Q)|^{-1}_{X_2}(I-P)|^{-1}_{Y_2}(I-P)\\
&=Q|^{-1}_{X_1}Q+(I-P)|^{-1}_{Y_2}(I-P)=I.
\end{align*}
where the second equation follows from $(I-Q)P|^{-1}_{Y_1}=0=P(I-Q)|^{-1}_{X_2}$ and the last equation by Lemma \ref{projections}. Thus, $ST=I$. To show $TS=I$, we have
\begin{align*}
TS&=P|^{-1}_{Y_1}Q^{-1}_{X1}Q(P+(I-Q))+(I-Q)|^{-1}_{X_2}(I-P)|^{-1}_{Y_2}(I-P)(P+(I-Q))\\
&=P|^{-1}_{Y_1}Q^{-1}_{X1}QP+(I-Q)|^{-1}_{X_2}(I-P)|^{-1}_{Y_2}(I-P)(I-Q)\\
&=P^{-1}|_{Y_1}P+(I-Q)|^{-1}_{X_2}(I-Q)=I
\end{align*}
where the last equation follows from Lemma \ref{projections}. Hence, we have that $T$ is the inverse of $S$ and $(x^{\sigma(i)}_i,x^{\sigma(i)*}_i)_{i=1}^\infty$ is an approximate Schauder frame.  We now check that $(x^{\sigma(i)}_i,x^{\sigma(i)*}_i)_{i=1}^\infty$ is an unconditonal approximate Schauder frame.  Let $x\in X$ and $\Gamma\subseteq\N$.  We have that
\begin{align*}
\left\|\sum_{i\in\Gamma}x^{\sigma(i)*}_i(x)x^{\sigma(i)}_i\right\|&\leq\left\|\sum_{i\in\Gamma\cap\sigma^{-1}(0)}x^{0*}_i(x)x^{0}_i\right\|+ \left\|\sum_{i\in\Gamma\cap\sigma^{-1}(1)}x^{1*}_i(x)x^{1}_i\right\|\\
&\leq C\left\|\sum_{i\in\N}x^{0*}_i(x)x^{0}_i\right\|+ C\left\|\sum_{i\in\N}x^{1*}_i(x)x^{1}_i\right\|\\
&=2C\|x\|.
\end{align*}
Thus, $(x^{\sigma(i)}_i,x^{\sigma(i)*}_i)_{i=1}^\infty$ is a $2C$-unconditional approximate Schauder frame.

For (iii)$\Rightarrow$(vi), we assume that for every
$\phi\in\{0,1\}^\N$, we have that
$(x^{\phi(i)}_i,x^{\phi(i)*}_i)_{i=1}^\infty$ is an unconditional
approximate Schauder frame. By Theorem \ref{T:weaklyWeaving} there
exists $E\geq 1$ such that for every $\phi\in\{0,1\}^\N$, we have
that $(x^{\phi(i)}_i,x^{\phi(i)*}_i)_{i=1}^\infty$ is a
$E$-unconditional approximate Schauder frame whose frame operator
$S$ satisfies $\|S\|,\|S^{-1}\|\leq E$.  Let $\sigma\in\{0,1\}^\N$
and $y\in [x^0_i]_{i\in \sigma^{-1}(0)}$.  We let $P$ be the basis
projection of $(x^0_i)_{i=1}^\infty$ onto $[x_i^0]_{i \in
\sigma^{-1}(0)}$ and $Q$ be the basis projection of
$(x^1_i)_{i=1}^\infty$ onto $[x_i^1]_{i \in \sigma^{-1}(0)}$.  Thus,
the frame operator of
$(x^{\sigma(i)}_i,x^{\sigma(i)*}_i)_{i=1}^\infty$ is given by
$S=P+(I-Q)$.  There exists unique $x\in X$ with $\|x\|\leq E\|y\|$
such that $S(x)=y$.  We have that
$$y=S(x)=(P+(I-Q))(Q+(I-Q))x=PQx+P(I-Q)x+(I-Q)x
$$
For the sake of contradiction, we assume that $(I-Q)x\neq 0$.  As,
$y\in[x_i^0]_{i\in\sigma^{-1}(0)}$, we have that $(I-P)y=0$.  Hence,
$0=(I-P)y=(I-P)(I-Q)x$.    Thus, $0=(Q+(I-P))(I-Q)x$ and hence the
operator $(Q+(I-P))$ is not invertible.  This is a contradiction
because $(Q+(I-P))$ is the frame operator of
$(x^{1-\sigma(i)}_i,x^{(1-\sigma(i))*}_i)_{i=1}^\infty$.  Thus,
$(I-Q)x=0$ and $x=Q(x)\in[x^1_i]_{i\in\sigma^{-1}(0)}$.  This
implies $P|_{[x^1_i]_{i\in\sigma^{-1}(0)}}\!\!:{[x_i^1]_{i \in \sigma^{-1}(0)}}\!\!\rightarrow {[x_i^0]_{i \in \sigma^{-1}(0)}}$ is invertible and
$\|P|_{[x^1_i]_{i\in\sigma^{-1}(0)}}^{-1}\|\leq E$.  Likewise,
$Q|_{[x^0_i]_{i\in\sigma^{-1}(0)}}\!\!:{[x_i^0]_{i \in \sigma^{-1}(0)}}\!\!\rightarrow {[x_i^1]_{i \in \sigma^{-1}(0)}}$ is invertible and
$\|Q|_{[x^0_i]_{i\in\sigma^{-1}(0)}}^{-1}\|\leq E$.
\end{proof}






\section{Perturbations}\label{S:5}

Perturbation theorems are powerful tools in both creating new coordinate systems and in showing that coordinate systems are resilient to error.
We first recall the following classical perturbation theorem for Schauder bases which is often referred to as the small perturbation lemma (see Theorem 6.18 in \cite{FHHMPZ}).

\begin{thm}\label{T:basisPert}
Let $(x_j^0)_{j=1}^\infty$ be a Schauder basis for a Banach space $X$ with biorthogonal functionals $(x^{0*}_j)_{j=1}^\infty$.  If $(x^1_j)_{j=1}^\infty$ is a sequence in $X$ such that $\sum_{j=1}^\infty \|x_j^0-x^1_j\|\|x^{0*}_j\|<1$ then $(x^1_j)_{j=1}^\infty$ is a Schauder basis for $X$ and is equivalent to $(x^0_j)_{j=1}^\infty$.

\end{thm}

Note that if $(x^1_j)_{j=1}^\infty$ is a perturbation of $(x^0_j)_{j=1}^\infty$ in the sense of Theorem \ref{T:basisPert}, then every weaving is also a perturbation of $(x^0_j)_{j=1}^\infty$ and hence $(x^1_j)_{j=1}^\infty$ and $(x^0_j)_{j=1}^\infty$ are woven.  In other words, a perturbation of a Schauder basis is a Schauder basis which is woven with it.  Perturbation theorems have been considered for Hilbert space frames and Schauder frames as well \cite{BCGLL, CLZ}.   However, we note that perturbation theorems are particularly natural in the context of approximate Schauder frames because when a Schauder frame is perturbed, one might expect the frame operator to be perturbed as well.  Hence, the perturbation of a Schauder frame would be expected to be an approximate Schauder frame, and indeed the name ``approximate Schauder frame" was chosen for this very reason.  We prove two perturbation theorems for approximate Schauder frames and prove they result in woven approximate Schauder frames.

\begin{thm}\label{T:Pert1}
Suppose $(x_i,f_i)_{i = 1}^\infty$ is a $C$-suppression unconditional approximate Schauder frame for $X$ with frame operator $S$. If $T$ is a bounded operator satisfying
\[
\|Id - T\| < C^{-1},
\]
then $(Tx_i, f_i)_{i = 1}^\infty$ is an unconditional approximate
Schauder frame for $X$ and is woven with $(x_i,f_i)_{i = 1}^\infty$.
\end{thm}

\begin{proof}
Set $(x_i^0, f_i^0)_{i=1}^\infty = (x_i, f_i)_{i=1}^\infty$ and $(x_i^1, f_i^1)_{i=1}^\infty = (Tx_i, f_i)_{i=1}^\infty$
for ease of notation. Let $x\in X$ and $\sigma\in\{0,1\}^\N$.  We
will first prove that the series $ \sum_{ i = 1}^\infty
f_i^{\sigma(i)} (x) x_i^{\sigma(i)} $ is unconditionally Cauchy and hence the frame
operator of $(x_i^{\sigma(i)}, f_i^{\sigma(i)})_{i = 1}^\infty$ is
well defined. Let $\vp>0$ and choose  $N \in \mathbb{N}$ so that for
all $J\subseteq \N$ with $\min(J) \geq N$, we have that $(1+\|T\|)\|\sum_{i\in
J}
f_i(x) x_i\|<\vp$.  For a fixed $J\subseteq \N$ with $\min(J) \geq N$, we have that
\begin{align*}
\bigg\|\sum_{i\in J} f^{\sigma(i)}_i(x) x^{\sigma(i)}_i\bigg\|&=\bigg\|\sum_{i\in
J\cap\sigma^{-1}(0)} f_i(x) x_i+\sum_{i\in J\cap\sigma^{-1}(1)}
f_i(x) T(x_i)\bigg\|\\
&\leq \bigg\|\sum_{i\in J\cap\sigma^{-1}(0)} f_i(x) x_i\bigg\|+\bigg\|\sum_{i\in
J\cap\sigma^{-1}(1)}
f_i(x) x_i\bigg\|\|T\|\\
&<\vp.
\end{align*}
Hence, the frame operator of $(x_i^{\sigma(i)}, f_i^{\sigma(i)})_{i
= 1}^\infty$ is well defined and converges unconditionally.

Let $\sigma \in \{0,1\}^\mathbb{N}$
and let $S_\sigma$ be the frame operator of
$(x^{\sigma(i)}_i,f_i^{\sigma(i)})_{i=1}^\infty$.  We will prove
that $\|Id-S_\sigma S^{-1}\|<1$, hence showing that $S_\sigma$ is
bounded with bounded inverse. If $x\in X$ and $y = S^{-1}x$, then
\begin{align*}
\left\|(Id-S_\sigma S^{-1})x\right\|&=\|Sy-S_\sigma
y\|\\
&=\left\|\sum_{i=1}^\infty f^0_i(y)x^0_i-\sum_{i=1}^\infty
f^{\sigma(i)}_i(y)x^{\sigma(i)}_i\right\|\\
&= \left\|\sum_{ i \in\sigma^{-1}(1)} \big[f_i (y) x_i-f_i (y) T(x_i)\big]\right\|\\
&\leq \left\|\sum_{i \in\sigma^{-1}(1)} f_i(y) x_i\right\| \|Id-T\| \\
&\leq C\|Sy\|\|Id-T\|\\
&=C\|x\|\|Id-T\|.
\end{align*}
Thus, $\|Id-S_\sigma S^{-1}\|\leq C\|Id-T\|<1$, and hence $S_\sigma$
is bounded with bounded inverse.
\end{proof}

Note that Theorem \ref{T:Pert1} is only stated for the case that $(x_i, f_i)_{i=1}^\infty$ is an unconditional approximate Schauder frame.  It is natural to ask if the same conclusion holds for conditional approximate Schauder frames as well.  This is not possible because if $(x_i, f_i)_{i=1}^\infty$ is a conditional approximate Schauder frame and $\alpha$  is any scalar other than $1$ then  $(x_i, f_i)_{i=1}^\infty$ is not woven with $(\alpha x_i, f_i)_{i=1}^\infty$.  However, the following perturbation result holds for both unconditional and conditional approximate Schauder frames.

\begin{thm}
If $(x_i^0, f_i^0)_{i = 1}^\infty$ is an approximate Schauder frame for $X$ with frame operator $S$ and $(x_i^1, f_i^1)_{i = 1}^\infty$ is a sequence in $X \times X^*$ satisfying
\[
\sum_{i = 1}^\infty \left(\|f_i^0 - f_i^1\|\|x_i^0\| + \|x_i^0 - x_i^1\|\|f_i^1\|\right) < \|S^{-1}\|^{-1}
\]
then $(x_i^1, f_i^1)_{i = 1}^\infty$ is an approximate Schauder frame for $X$ that is woven with $(x_i^0, f_i^0)_{i = 1}^\infty$.
\end{thm}

\begin{proof}

Let $\sigma\in\{0,1\}^\N$ and $x\in X$.  We first show that the
series $ \sum_{ i = 1}^\infty f_i^{\sigma(i)} (x) x_i^{\sigma(i)} $
is Cauchy and hence the frame operator of $(x_i^{\sigma(i)},
f_i^{\sigma(i)})_{i = 1}^\infty$ is well defined. Let $\vp>0$ and
choose  $N \in \mathbb{N}$ so that for all $m
> n \geq N$,
\[
\left\|\sum_{i = n}^m f_i^0(x)x_i^0 \right\| + \sum_{i = n}^m
(\|f_i^0 - f_i^1\|\|x_i^0\| + \|x_i^0 - x_i^1\|\|f_i^1\|)\|x\| <
\epsilon.
\]
Thus,
\begin{align*}
\left\| \sum_{i = n}^m f_i^{\sigma(i)} (x) x_i^{\sigma(i)} \right\| &\leq \left\| \sum_{i = n}^m f_i^0(x) x_i^0 \right\| + \left\|\sum_{i \in [n,m] \cap \sigma^{-1}(1)} \big[f_i^0(x) x_i^0 - f_i^1(x) x_i^1\big] \right\|\\
&\leq \left\| \sum_{i = n}^m f_i^0(x) x_i^0 \right\| +\sum_{i=n}^m
\big[\left\|f_i^0(y) x_i^0 - f_i^1(y) x_i^0
\right\|+\left\|f_i^1(x) x_i^0 - f_i^1(x) x_i^1 \right\|\big]\\
&\leq \left\| \sum_{i = n}^m f_i^0(x) x_i^0 \right\| + \sum_{i =
n}^m (\|f_i^0 - f_i^1\|\|x_i^0\| + \|x_i^0 - x_i^1\|\|f_i^1\|)\|x\|\\
&<\epsilon.
\end{align*}
Hence, the frame operator of $(x_i^{\sigma(i)}, f_i^{\sigma(i)})_{i
= 1}^\infty$ is well defined.


Let $\sigma \in \{0,1\}^\mathbb{N}$ and let $T$ be the frame
operator of $(x^{\sigma(i)}_i,f_i^{\sigma(i)})_{i=1}^\infty$.  We
will prove that $\|Id-TS^{-1}\|<1$, hence showing that $T$ is bounded
with bounded inverse.  If $x\in X$ and $y = S^{-1}x$, then
\begin{align*}
\left\|(Id-TS^{-1})x\right\|&=\|Sy-Ty\|\\
&=\left\|\sum_{i=1}^\infty f^0_i(y)x^0_i-\sum_{i=1}^\infty
f^{\sigma(i)}_i(y)x^{\sigma(i)}_i\right\|\\
&= \left\|\sum_{ i \in\sigma^{-1}(1)}\big[ f_i^{0} (y) x_i^{0}-f_i^{1} (y) x_i^{1}\big]\right\|\\
&\leq \sum_{i = 1}^\infty \|f_i^0(y) x_i^0 - f_i^1(y) x_i^1 \|\\
&\leq \sum_{i=1}^\infty \|f_i^0(y) x_i^0 - f_i^1(y) x_i^0
\|+\|f_i^1(y) x_i^0 - f_i^1(y) x_i^1 \|\\
&\leq \left(\sum_{i=1}^\infty \|f_i^0-f^1_i\|\| x_i^0\|+\|x_i^0 - x_i^1 \|\|f_i^1\|\right)\|y\|\\
&\leq \left(\sum_{i=1}^\infty \|f_i^0-f^1_i\|\| x_i^0\|+\|x_i^0 -
x_i^1 \|\|f_i^1\|\right)\|S^{-1}\|\|x\|.
\end{align*}
Thus, $\|Id-TS^{-1}\|\leq\left(\sum_{i=1}^\infty \|f_i^0-f^1_i\|\|
x_i^0\|+\|x_i^0 - x_i^1 \|\|f_i^1\|\right)\|S^{-1}\|<1$ proving that $T$
is bounded with bounded inverse.
\end{proof}

\end{document}